\documentclass[10pt,a4paper]{article}
\usepackage{amsmath,amssymb,amsthm,esint,bm}
\usepackage{mathrsfs}
\usepackage{bookmark}
\allowdisplaybreaks[3]

\usepackage{xcolor}

\usepackage[utf8]{inputenc}   
\usepackage[T1]{fontenc}      
\usepackage{lmodern}

\usepackage{citeref}
\usepackage[notref,notcite]{showkeys}

\newcommand{\ainc}[1]{\hyperref[ainc]{{\normalfont(aInc){\ensuremath{_{#1}}}}}}
\newcommand{\adec}[1]{\hyperref[adec]{{\normalfont(aDec){\ensuremath{_{#1}}}}}}
\newcommand{\inc}[1]{\hyperref[inc]{{\normalfont(Inc){\ensuremath{_{#1}}}}}}
\newcommand{\dec}[1]{\hyperref[dec]{{\normalfont(Dec){\ensuremath{_{#1}}}}}}

\allowdisplaybreaks

\usepackage[T1]{fontenc}
\usepackage[utf8]{inputenc}
\usepackage{authblk}
\usepackage{amsmath,amssymb,amsthm,esint,bm}
\usepackage{mathrsfs}
\usepackage{bookmark}
\usepackage{amsmath}

\newtheorem{theorem}{Theorem}[section]
\newtheorem{lemma}[theorem]{Lemma}

\theoremstyle{definition}
\newtheorem{remark}[theorem]{Remark}
\newtheorem{definition}[theorem]{Definition}

\numberwithin{equation}{section}

\newcommand{\R}{\mathbb{R}}

\newcommand{\ls}{\leqslant}
\newcommand{\rs}{\geqslant}

\newcommand{\holder}{H{\"o}lder}

\title{ Global Weighted Gradient Estimates for  Double Obstacle Problems with Orlicz Growth  }
\author[a]{Qi Xiong}
\author[b]{Xing Fu\thanks{Corresponding author.}}

\affil[a]{School of Mathematics, Southwest Jiaotong University, Chengdu, Sichuan, 610031,  China}
\affil[b]{Hubei Key Laboratory of Applied Mathematics,
	Faculty of Mathematics and Statistics, Hubei University, Wuhan 430062, China}

\usepackage{hyperref}
\begin{document}
\arraycolsep=1pt
\maketitle
\footnotetext[1]{E-mail: xq@swjtu.edu.cn(Q. Xiong),  xingfu@hubu.edu.cn (X. Fu).}

\begin{abstract}
We study an irregular double obstacle problem with Orlicz growth on a bounded nonsmooth domain. The nonlinear operator is assumed to satisfy standard monotonicity and   growth conditions, together with a small bounded mean oscillation condition in the spatial variable, while the domain is Reifenberg flat. We establish a global weighted Orlicz estimate for the gradient of the weak solution in terms of the nonhomogeneous datum, the gradients of the two obstacles, and the boundary datum.

\medskip


\end{abstract}


\section{Introduction and main results}

 \hskip\parindent  The classical obstacle problem seeks to find minimizers of an energy functional 
 subject to a constraint representing a physical barrier. This mathematical model 
 is naturally formulated within the framework of variational inequalities, a field 
 originating from the calculus of variations and nonlinear partial differential equations (PDEs). 
 Over the past decades, the theory of variational inequalities has proven to be an 
 indispensable tool in various applied sciences, such as fluid filtration in porous 
 media, elasto-plasticity, and optimal control; see, for example, \cite{bsy,bsy2,cho,kl,rt1,s26,s28,xiong1,xiong2,xiong3}.

 In this paper, we consider double obstacle problems with Orlicz growth.
Let $\Omega \subset \mathbb{R}^{n}, n \geq 2$, be a bounded open domain. Assume that   $\psi_{1}, \psi_{2} \in W^{1, G}(\Omega)$ are obstacle functions   and $h\in W^{1, G}(\Omega)$ is a given boundary data with $\psi_{1} \leq h \leq \psi_{2}$ a.e. in $\Omega$, where $W^{1, G}(\Omega)$ that will be specified later in Section 2 with an $N$-function $G$ defined in \eqref{g}.
Define a convex admissible set
$$
\mathcal{A}_{h}(\Omega)=\left\{\varphi \in W_{h}^{1, G}(\Omega): \psi_{1} \leq \varphi \leq \psi_{2} \text { a.e. in } \Omega\right\} .
$$

We deal with a weak solution $u \in \mathcal{A}_{h}(\Omega)$ to the double obstacle problem which means that it satisfies the variational inequality
 
\begin{equation}\label{u}
	\int_{\Omega}  a(x, D u) \cdot D(v-u) \mathrm{d} x \geqslant \int_{\Omega} b(x, F ) \cdot D(v-u) \mathrm{d} x  
\end{equation}
for all $v \in \mathcal{A}_{h}(\Omega)$. Here we
assume that $F \in L^G(\Omega;\mathbb{R}^n)$ and  $a=a(x,\eta): \Omega \times\mathbb{R}^n \rightarrow \mathbb{R}^n$ is measurable for every $\eta \in \mathbb{R}^n$  and differentiable for almost every $x\in \Omega$  and there exist constants $0<l\leqslant 1 \leqslant L<+\infty$ such that for all $x \in \Omega,\eta,\lambda \in \mathbb{R}^n$,   
\begin{eqnarray} \label{u1}
	\left\{\begin{array}{r@{}c@{}ll}
		&&D_{\eta} a(x,\eta )\lambda \cdot \lambda \geqslant l\dfrac{g(|\eta|)}{|\eta|}|\lambda|^2 \,, \\[0.05cm]
		&&|a(x,\eta)|+|\eta||D_{\eta} a(x,\eta )|\leqslant Lg(|\eta|)\,, \\[0.05cm]
	\end{array}\right.
\end{eqnarray}
where $D_{\eta}$ denotes the differentiation in $\eta$.  
Moreover, let $b(x,\eta): \Omega \times\mathbb{R}^n \rightarrow \mathbb{R}^n$  satisfy
\begin{equation}\label{u2}
	|b(x,\eta)|\leqslant L g(|\eta|).
\end{equation}
Here $g(t) : [0,+\infty)\rightarrow [0,+\infty)$ satisfies
\begin{eqnarray}\label{u3}
	\left\{\begin{array}{r@{}c@{}ll}
		&&g(t)=0 \ \ \ \Leftrightarrow \ \ \   t=0 \,, \\[0.05cm]
		&&g(\cdot)\in C^{1}(\mathbb{R}^+)\,, \\[0.05cm]
		&&   i_{g}=: \inf_{t>0}\frac{tg'(t)}{g(t)}\leq \sup_{t>0}\frac{tg'(t)}{g(t)}=:s_{g}<\infty, \ \ \  where \ \  0< i_{g} \leqslant 1 \leqslant s_{g} <\infty. \, \\[0.05cm]
	\end{array}\right.
\end{eqnarray}
Define
\begin{equation}\label{g}
	G(t):= \int_0^tg(\tau)\,d\tau  \ \ \ \mbox{for}\ \ t\geq0.
\end{equation}

The existence and uniqueness of weak solutions to the variational inequality \eqref{u} were established in \cite{rt1}. Over the past few decades, obstacle problems have been extensively investigated; we refer to \cite{xiong1, xiong2, xiong3, xiong4, xiong5, xiong6} for   gradient estimates, and to \cite{l2, rt1} for H\"older regularity results.

The aim of this paper is to establish global weighted
Calder\'on--Zygmund-type estimates for the double obstacle problem
\eqref{u}. Calder\'on--Zygmund theory is a classical regularity theory
which, broadly speaking, asserts that the regularity of a solution is
determined by that of the given data, including the nonhomogeneous term,
the coefficients, the obstacles, and the boundary of the domain; see \cite{ming1,ming2}. Since the celebrated Calder\'on--Zygmund estimate was first
introduced in \cite{cz1}, it has played a fundamental role in the regularity
theory of elliptic and parabolic equations. 
A fundamental contribution to nonlinear gradient estimates was made by
Iwaniec in \cite{i1}, where a local Calder\'on--Zygmund estimate was established
for a nonhomogeneous elliptic $p$-Laplacian equation. DiBenedetto and
Manfredi extended this result to elliptic systems in \cite{ddm}. Subsequently,
Caffarelli and Peral generalized the estimate in \cite{i1} to a broader class
of elliptic equations of $p$-Laplacian type in \cite{cpp}. These results were later extended to the parabolic setting by
Acerbi and Mingione in \cite{am3}. Moreover, Byun and Wang studied a general elliptic equation whose
nonlinearity $a(x,Du)$ satisfies the standard ellipticity and
$p$-growth conditions and established the corresponding
Calder\'on-Zygmund estimate under a suitable smallness assumption on
the BMO oscillation of $a$ with respect to the spatial variable $x$;
see \cite{bw}.  Recently, there have been active research activities on the regularity
theory for general nonlinear elliptic problems. In particular, Verde
\cite{v} proved global Calder\'on--Zygmund estimates over the whole domain
$\mathbb{R}^{n}$ for weak solutions to the elliptic system 
\begin{equation}\label{16}
	-\operatorname{div}\left( \frac{g(|Du|)}{|Du|}Du\right) 
	=
	-\operatorname{div}\left( \frac{g(|DF|)}{|DF|}DF\right),
\end{equation}
where $g(t)$ satisfies \eqref{u3}. Yao and Zhou \cite{yz1}  obtained local
$L^{q}$-estimates for every $q\geq 1$ for weak solutions of \eqref{16}.
Furthermore, Byun and Cho \cite{bc} extended this result to a global gradient
estimate in the framework of general Orlicz spaces on a Reifenberg
domain. Moreover,  Cho \cite{cho1} established   global Calder\'on-Zygmund estimates for  the
zero-Dirichlet problem involving general nonlinear operators with  Orlicz growth. 
 
In obstacle problems, it is well known that the regularity of solutions is generally limited by the smoothness of the obstacles. For problems involving general nonlinear operators of $p$-Laplace type over nonsmooth domains, global Calder\'on-Zygmund estimates were established by Byun, Cho, and Wang \cite{bcw}, while the corresponding weighted estimates were subsequently developed in \cite{bcp}. Furthermore, such regularity estimates have been extended to the parabolic setting by B{\"o}gelein, Duzaar, and Mingione \cite{bdm}. More recently, Cho \cite{cho} extended these results to global weighted Orlicz estimates for obstacle problems with Orlicz growth over bounded nonsmooth domains. Regarding double obstacle problems, Byun and Ryu \cite{br} established local Calder\'on-Zygmund estimates for $p$-Laplace type operators. Subsequently, Byun, Liang, and Ok \cite{blo} extended these results to obstacle problems under Orlicz growth with zero Dirichlet boundary data over bounded nonsmooth domains. 

In the present paper, our main objective is to establish global weighted estimates for double obstacle problems with non-zero Dirichlet boundary data over Reifenberg domains. To overcome the difficulties arising from the non-zero boundary data, we first construct a suitable comparison estimate that reduces the non-homogeneous problem to a homogeneous one. Through a sequence of further comparison arguments, the double obstacle problem is then transformed into a reference equation with zero Dirichlet boundary conditions. Finally, the main result is achieved by employing a Vitali-type covering lemma.

We now proceed to outline our main results, starting with the presentation of key definitions, notations, and assumptions.

Let 
\[
\theta(a, B_\varrho(x_0))(x) := \sup_{\eta \in \mathbb{R}^n \setminus \{0\}} \frac{|a(x, \eta) - \overline{a}_{B_\varrho(x_0)}(\eta)|}{g(|\eta|)}
\]
where 
\[
\overline{a}_{B_\varrho(x_0)}(\eta) := \fint_{B_\varrho(x_0)} a(x, \eta) \,  {d}x.
\]
 \begin{definition}
   We say that $(a, \Omega)$ is $(\delta, R)$-vanishing for some $\delta,R > 0$, if
\[
\sup_{x_0 \in \mathbb{R}^n} \sup_{0 < \varrho < R} \fint_{B_\varrho(x_0)} \theta(a, B_\varrho(x_0))(x) \,  {d}x \leq \delta,
\]
and if for each $r \in (0, R]$ and each $x \in \partial\Omega$ there exists a coordinate system $\{y_1, \cdots, y_n\}$ such that in this coordinate system the origin is at $x$ and it satisfies
\[
\{y \in B_r(0) : y_n > \delta r\} \subset B_r(0) \cap \Omega \subset \{y \in B_r(0) : y_n > -\delta r\}. 
\]
 \end{definition}
\begin{remark}\label{rei}
	For every $x \in \partial \Omega$ and $r \in (0,R]$ we choose a coordinate system $(z_1,\dots,z_n)$ with the origin at some interior point of $  \Omega$ such that in this coordinate system $x=- \delta re_n$ and
	$$
	B_{\frac{7}{8}r}(0)\cap \{z_n>0\} \ \subset\ \Omega\cap B_{\frac{7}{8}r}(0)\ \subset\ B_{\frac{7}{8}r}(0)\cap \{z_n>-2\delta r\}.
	$$

	Note that a domain with Lipschitz boundary with Lipschitz seminorm $\delta\in(0,1)$
	is $(\delta,R)$-Reifenberg flat for some $R>0$ and that the ball $B_r$ is $(\delta,2\delta r)$-Reifenberg flat for any $\delta\in (0,\frac{1}{2})$.
\end{remark}

The main result of this paper reads as follows.
 \begin{theorem}\label{weight}
 	Given  an N-function $\Phi \in \Delta_2 \cap\nabla_2$, let $\omega$ be a weight in $A_{\Phi}$. Assume that $|F|\in L^{\Phi\circ G}_{\omega}(\Omega), |D\psi_i|\in L^{\Phi\circ G}_{\omega}(\Omega), i=1,2,$ $h\in W^{1,G}(\Omega)$ is a given boundary datum with $\psi_1 \leq h \leq \psi_2$ a.\,e. in $\Omega$ and $ |Dh|\in L^{\Phi\circ G}_{\omega}(\Omega)$. Then there exist a small positive constant
 	$
 	\delta=\delta(n,l,L,i_g,s_g, \Phi,\omega, \Omega)
 	$
 	and a positive constant
 	$
 	c=c(n,l,L,i_g,s_g, \Phi,\omega, \Omega)
 	$
 	such that, if $(a,\Omega)$ satisfies the $(\delta,R)$-vanishing
 	condition,
 and $u \in \mathcal{A}_{h}(\Omega)$ solves the variational inequality \eqref{u}, then
 	\begin{eqnarray}\label{uff}\nonumber
 		\int_{\Omega}(\Phi\circ G)(|Du|)\omega(x)dx& \ls& c\int_{\Omega}(\Phi\circ G)(|F|)\omega(x)dx+c\int_{\Omega}(\Phi\circ G)(|D\psi_1|)\omega(x)dx\\
 		&&\hspace{0.1cm}+c\int_{\Omega}(\Phi\circ G)(|D\psi_2|)\omega(x)dx+c\int_{\Omega}(\Phi\circ G)(|Dh|)\omega(x)dx,
 	\end{eqnarray}
 	where $G$ is defined as in \eqref{g}.
 \end{theorem}

\section{   Preliminaries} 
\hskip\parindent In this section, we introduce several definitions and preliminary results. We begin by defining some notations for convenience. Let $x \in \mathbb{R}^n$ and $r > 0$.
\begin{itemize}
	\item $\Omega_r(x) := \Omega \cap B_r(x) \ \ \ \&  \ \ \ \Omega_r := \Omega \cap B_r \ \ \ \& \ \ \ B_r := B_r(0).$ 
	\item $B_r^+(x) := \{y \in B_r(x) : y_n > 0\} \ \ \   \& \ \ \ B_r^+ := B_r^+(0)  \ \ \  \& \ \ \ T_r := \{x \in B_r : x_n = 0\}.$ 
\end{itemize}

\begin{definition}
	A function $B :[0,+\infty)\rightarrow[0,+\infty)$ is called a Young function if it is convex and $B(0)=0$.
\end{definition}
\begin{definition}
	Assume that B is a Young function,  the Orlicz class $K^{B}(\Omega)$ is the set of all measurable functions $f : \Omega\rightarrow\mathbb{R}$ satisfying
	\begin{equation*}
		\int_\Omega B(|f|) \operatorname{d}\!\xi < \infty.\nonumber
	\end{equation*}
	The Orlicz space $L^{B}(\Omega)$ is the linear hull  of the Orlicz class  $K^{B}(\Omega)$ with the Luxemburg norm
	\begin{equation*}
		\Vert f \Vert_{L^B(\Omega)}:=\inf\left\lbrace \alpha>0: \ \ \int_{\Omega}B\left(\frac{|f|}{\alpha} \right) \operatorname{d}\!\xi \leqslant1\right\rbrace .
	\end{equation*}
	Furthermore, the Orlicz-Sobolev space $W^{1,B}(\Omega)$ is defined as
	\begin{equation*}
		W^{1,B}(\Omega)=\left\lbrace  f\in L^{B}(\Omega)\cap W^{1,1}(\Omega) \ \vert \ Df\in L^{B}(\Omega)\right\rbrace.\nonumber
	\end{equation*}
	Here, $D$ stands for gradient. The space $W^{1,B}(\Omega)$, equipped with the norm
	$\Vert f \Vert_{W^{1,B}(\Omega)}:=\Vert f \Vert_{L^B(\Omega)}+\Vert Df \Vert_{L^B(\Omega)},$ is a Banach space. Clearly, $W^{1,B}(\Omega)=W^{1,p}(\Omega)$, the standard Sobolev space, if $B(t)=t^p$ with $p\geqslant1$.
\end{definition}

The subspace $W_{0}^{1,B}(\Omega)$ is the closure of $C_{0}^{\infty}(\Omega)$ in $W^{1,B}(\Omega)$. The above properties about Orlicz space can be found in \cite{hphp1,rr28}.

 \begin{definition}
 	A Young function $B$ is called an $N$-function if
 	$$0<B(t)<+\infty \ \ for \ t>0$$
 	and
 	\begin{equation*}\label{nhanshu}
 		\lim_{t\rightarrow+\infty}\frac{B(t)}{t}=\lim_{t\rightarrow0}\frac{t}{B(t)}=+\infty.
 	\end{equation*}
 We observe that $ G(t)$, as defined in \eqref{g}, is an 
 $N$-function.

 \end{definition}
 
 \begin{definition}
 	A Young function B is said to satisfy the global $\vartriangle_2$ condition, denoted by $B\in\vartriangle_2$, if there exists a positive constant C such that for every $t>0$,
 	\begin{equation*}
 		B(2t)\leq CB(t).
 	\end{equation*}
 	Similarly, a Young function B is said to satisfy the global $\bigtriangledown_2$ condition, denoted by $B\in\bigtriangledown_2$, if there exists a  constant $\theta >1$ such that for every $t>0$,
 	\begin{equation*}
 		B(t)\leq \frac{B(\theta t)}{2\theta}.
 	\end{equation*}
 \end{definition}
 	The Young conjugate  of a Young function B will be denoted by $B^{\ast}$ and defined as
 $$B^{\ast}(t)=\sup_{s\geq 0}\left\lbrace st-B(s)\right\rbrace  \ \ for \ t\geq 0.$$
 We next recall Young's inequality for $N$-functions, which will be used throughout the paper.
 \begin{lemma}\cite{yz1}\label{gyoung}
 If $B\in \bigtriangleup_{2}\cap \bigtriangledown_{2}$ is an
 	$N$-function, then, for every $\varepsilon>0$, there exists a constant
 	$c(\varepsilon)>0$ such that
 	\[
 	st\leq \varepsilon B^{*}(s)+c(\varepsilon)B(t)
 	\qquad \text{for all } s,t\geq0.
 	\]
 \end{lemma}
 Note that $G(t)$, as defined in \eqref{g}, satisfies   Young's inequality.
 
We   recall the following standard estimate for the Young conjugate; see \cite{a1}:
 \begin{equation}\label{a(x)4}
 	B^{*}\left( \frac{B(t)}{t}\right) \leqslant B(t).
 \end{equation}

Define an auxiliary vector field $V : \mathbb{R}^n \to \mathbb{R}^n$ given by
\begin{align*}
	V(\xi) := \left( \frac{g(|\xi|)}{|\xi|} \right)^{1/2} \xi, \quad \text{for each } \xi \in \mathbb{R}^n.
\end{align*}

We recall the following standard relations among $a$, $V$, and $G$, see Lemma 3 in \cite{de1}.
\begin{lemma}\label{vg}
Let $a$ and $G$ satisfy \eqref{u1} and \eqref{u3}, respectively. Then
\begin{align*}
	[a(x, \xi) - a(x, \eta)] \cdot (\xi - \eta)  \approx |V(\xi) - V(\eta)|^2, 
	\end{align*}
	\begin{align*}
		a(x, \xi) \cdot \xi \approx |V(\xi)|^2 \approx G(|\xi|)
	\end{align*}
 and
 \begin{align}\label{cv}
G(|\xi-\eta|)\leqslant c_v \left[ |V(\xi) - V(\eta)|^2+G(\eta)\right] 
	\end{align}
uniformly in $x, \xi, \eta \in \mathbb{R}^n$. 
 Here $c_v>1$ and the hidden constants in the above relations depend only on $n,l,L,i_g,s_g$.
\end{lemma}

\begin{df}
We say that $\omega$ is a weight if $\omega$ is a locally integrable positive function in $\R^n$ and write $\omega(U)$ for Lebesgue measure of any measurable set $U \subset \R^n$, given by
$$\omega(U):=\int_{U}\omega(x)dx.$$
\end{df}

\begin{df}
For $1<q<\infty$, the Muckenhoupt class $A_q$ consists of all
weights $\omega$ satisfying
\begin{align}
	[\omega]_{A_q}
	:=
	\sup_{B_r(y)\subset\mathbb{R}^n}
	\left(\fint_{B_r(y)}\omega(x)\,dx\right)
	\left(
	\fint_{B_r(y)}
	\omega(x)^{-\frac{1}{q-1}}\,dx
	\right)^{q-1}
	<\infty.
\end{align}
The quantity $[\omega]_{A_q}$ is called the $A_q$ characteristic
of $\omega$.

For an N-function $\Phi$, a weight $\omega$ belongs to an $A_{\Phi}$ class if, for each $\rho>0$ and all balls $B_r(y)$,
$$\left( \fint_{B_r(y)}\rho \omega(x)dx \right) \Phi'\left( \fint_{B_r(y)}(\Phi')^{-1}\left( \frac{1}{\rho\omega(x)}\right)dx \right) \ls c_{\Phi},$$
where $c_{\Phi} \in (0,\infty)$ is independent of $\rho$ and $B_r(y)$. The weighted Orlicz space $L^{\Phi}_{\omega}(\Omega)$ related to $A_{\Phi}$ consists of all measurable functions $f: \Omega \rightarrow \R$ such that
$$\int_{\Omega} \Phi(|f|)\omega(x)dx<\infty.$$
We then define the Luxemburg norm of the weighted Orlicz space by
$$||f||_{L^{\Phi}_{\omega}(\Omega)}=\inf \left\lbrace M>0: \int_{\Omega}\Phi \left( \frac{|f(x)|}{M}\right)\omega(x)dx \ls 1 \right\rbrace. $$
\end{df}

Next define lower and upper indices of an N-function $\Phi$. Let
$$h_{\Phi}(M):=\sup_{t>0}\frac{\Phi(Mt)}{\Phi(t)}, \ \ \ \forall\, M>0.$$
The lower index $i(\Phi)$ and the upper index $I(\Phi)$ are defined, respectively, by
$$i(\Phi):=\sup_{0<M<1}\frac{\log h_{\Phi}(M)}{\log M}\quad \&\quad
I(\Phi):=\inf_{1<M<\infty}\frac{\log h_{\Phi}(M)}{\log M}.$$
  Clearly, $1\leq i(\Phi) \leq I(\Phi)$.

Then we recall two lemmas from \cite[Chapter 7]{g1}.
\begin{lemma}\label{weight1}
The class \( A_q \) is monotone increasing with respect to the index \( q \); that is,
\[
A_{q_1} \subset A_{q_2} \quad \text{whenever } q_1 \leq q_2.
\]
Furthermore, it satisfies the following open-end property: for any \( 1 < q < \infty \),
\[
A_q = \bigcup_{s \in (1,q)} A_s.
\]
\end{lemma}
\begin{lemma}\label{weight2}
Suppose \( \omega \in A_q \) for some \( 1 < q < \infty \),  let \( B \subset \mathbb{R}^n \) be a ball and let \( L \subset B \) be a measurable subset. Then there exist constants \( \nu, d > 0 \), depending only on \([ \omega ]_q\) and \( n \), such that
\[
\frac{1}{d} \left( \frac{|L|}{|B|} \right)^q \leq \frac{\omega(L)}{\omega(B)} \leq d \left( \frac{|L|}{|B|} \right)^{\nu}.
\]
\end{lemma}

Next, we briefly recall the definition and basic properties of the Hardy-Littlewood maximal operator. For a locally integrable function $f$ on $\mathbb{R}^n$, the maximal function $\mathcal{M}f$ is defined by
\begin{align*}
	\mathcal{M}f(x)
	&= \sup_{r>0} \fint_{B_r} |f(y)|\,\mathrm{d}y,
	\qquad \text{for every } x\in\mathbb{R}^n.
\end{align*}
If $f$ is defined only on a bounded domain $U \subset \mathbb{R}^n$, we naturally identify it with its zero extension outside $U$ by setting
\begin{align*}
	\mathcal{M}f &= \mathcal{M}(f\chi_U),
\end{align*}
where $\chi_U$ denotes the characteristic function of $U$. It is a classical result that  $\mathcal{M}$ satisfies the following standard weak $(1,1)$ estimate:
\begin{align}\label{one}
	\left|\left\{ x\in\mathbb{R}^n : \mathcal{M}f(x) > K \right\}\right|
	&\leq \frac{c}{K} \int_{\mathbb{R}^n}|f(x)|\,\mathrm{d}x,
	\qquad \text{for every } K>0,
\end{align}
where $c = c(n) > 0$ is a universal constant.

According to \cite[Theorem 2.1.1]{kk}, we have the following lemma.
\begin{lemma}\label{weight3}
 Let $\Phi$ be an $N$-function satisfying the $\Delta_2$ and $\nabla_2$ conditions, and let $\omega$ be a weight on $\mathbb{R}^n$. Then the following statements are equivalent:
\begin{enumerate}
    \item[(i)] There exists a positive constant \( c_1 \) such that, for every \( f \in L^\Phi_\omega(\mathbb{R}^n) \),
    \[
    \int_{\mathbb{R}^n} \Phi(\mathcal{M}f)\,\omega(x)\,dx \leq c_1 \int_{\mathbb{R}^n} \Phi(|f|)\,\omega(x)\,dx.
    \]

    \item[(ii)] There exists a positive constant \( c_2 \) such that, for every \( \rho > 0 \) and \( f \in L^\Phi_{\rho\omega}(\mathbb{R}^n) \),
    \[
    \|\mathcal{M}f\|_{L^\Phi_{\rho\omega}(\mathbb{R}^n)} \leq c_2 \|f\|_{L^\Phi_{\rho\omega}(\mathbb{R}^n)}.
    \]

    \item[(iii)] \( \omega \in A_\Phi \).

    \item[(iv)] \( \omega \in A_{i(\Phi)} \).
\end{enumerate}
\end{lemma}

We next recall a classical characterization of weighted Orlicz spaces in terms of the weighted measures of superlevel sets, which will be essential in the subsequent iteration argument; see \cite{rr28}.

\begin{lemma}\label{weight4}
Let $U \subset \mathbb{R}^n$ be a bounded domain, and assume that $f$ is a measurable function compactly supported in $U$. Then, for any $\sigma > 0$, $K \geq 1$, and $N$-function $\Phi$ satisfying the $\Delta_2$ and $\nabla_2$ conditions, we have
\begin{align}
	f \in L_{\omega}^{\Phi}(U) \quad \Longleftrightarrow \quad 
	S := \sum_{i=1}^{\infty} \Phi(K^i) \omega\bigl(\{x \in U : |f(x)| > \sigma K^i\}\bigr) < \infty.
\end{align}
Moreover, there exists a constant $c = c(\sigma, K, \Phi) > 0$ such that the following estimate holds:
\begin{align}
	\frac{1}{c}S \leq \int_U \Phi(|f|) \omega(x) \, dx \leq c\bigl(\omega(U) + S\bigr).
\end{align}
\end{lemma}
The following weighted covering lemma will be a key ingredient in the subsequent good-$\lambda$ argument. Its proof can be found in \cite[Lemma 3.8]{mp}. 
 
 \begin{lemma}\label{weight5}
 Let $\Omega \subset \mathbb{R}^n$ be a bounded $(\delta,R)$-Reifenberg flat domain with $0 < \delta < \frac{1}{8}$, and assume that the weight $\omega$ belongs to the Muckenhoupt class $A_q$ for some $q \in (1, \infty)$. Suppose that $\Omega$ is covered by a finite collection of balls $\{B_r(z_j)\}_{j=1}^{l}$ centered in $\overline{\Omega}$ with radius $0 < r \leq \frac{R}{2000}$. For any measurable sets $C$ and $D$ satisfying $C \subset D \subset \Omega$, if there exists a small constant $\varepsilon > 0$ such that
 \begin{align*}
 	\omega(C) &< \varepsilon \, \omega(B_r(z_j)) \qquad \text{for all } j=1,\ldots,l,
 \end{align*}
 and for every $\rho \in (0, 2r]$ and any ball $B_\rho(y)$ centered in $\Omega$, we have
 \begin{align*}
 	B_\rho(y) \cap \Omega \subset D \qquad \text{whenever} \qquad \omega(C \cap B_\rho(y)) \geq \varepsilon \, \omega(B_\rho(y)),
 \end{align*}
 then the following measure estimate holds:
 \begin{align*}
 	\omega(C) &\leq \varepsilon_1 \omega(D),
 \end{align*}
 where $\varepsilon_1 = 20^{nq}[\omega]_{A_q}^2 \varepsilon$.
 \end{lemma}

  \section{Comparison estimates}
 In this section, we establish a sequence of comparison estimates by introducing several auxiliary problems associated with the original double obstacle problem. These estimates will be used in the next section to derive the global weighted Orlicz estimate.
 
We first assume that
\begin{equation}\label{ass1}
  \sup_{0<\varrho<5} \fint_{B_\varrho} \theta(a, B_\varrho)(x) \, \mathrm{d}x \leq \delta,
\end{equation}

\begin{equation}\label{ass2}
\fint_{\Omega_5} G(|Du|) \, \mathrm{d}x \leq 1, \qquad \fint_{\Omega_5}\left[  G(|F|)+G(|Dh|)+G(|D\psi_1|)+G(|D\psi_2|)\right]  \, \mathrm{d}x \leq \delta,
\end{equation}
where $\delta \in (0, \frac{1}{8})$ is determined later universally.

 The arguments in this subsection shall need the condition
\begin{align}\label{ass3}
 B_5^+ \subset \Omega_5 \subset \left\{ x \in B_5 : x_n > -\frac{80}{7}\delta \right\},
\end{align}
where $\delta$ is as in \eqref{ass1} and \eqref{ass2}.
\begin{lemma}\label{uw1}
	 Assume that \eqref{u1}-\eqref{u3} are satisfied and   $u \in \mathcal{A}_{h}(\Omega)$   satisfies the variational inequality \eqref{u}.
	 For any $\varepsilon > 0$ there exists a small $\delta=\delta(n,l,L,i_g,s_g,\varepsilon) \in (0,\frac{1}{8})$ such that if \eqref{ass2} holds for such $\delta$ and  $w_1 \in W^{1,G}(\Omega_5)$ with $\psi_1-h\leqslant w_1\leqslant \psi_2-h $ a.e. in $\Omega_5$ solves
	 \begin{equation}\label{w1}
	 	\left\{\begin{array}{r@{\ \ }c@{\ \ }ll}
	 		&\int_{\Omega_5}  a(x,  Dw_1) \cdot D(v-w_1)dx \geqslant \int_{\Omega_5}  b(x, F) \cdot D(v-w_1)dx,  \\[0.05cm]
	 		&w_1=u-h \ \ \mbox{on}\ \ \partial \Omega_5 \,,
	 	\end{array}\right.
	 \end{equation}
	where $v\in w_1+W_0^{1,G}(\Omega_5)$ satisfies $\psi_1-h \leqslant v \leqslant \psi_2-h$ a.e. in $\Omega_5$.
	Then
	\begin{equation}\label{vuw1}
		\fint_{\Omega_5} G(|Dw_1|)\,dx \leq c \quad \text{and} \quad \fint_{\Omega_5} |V(Du)-V(Dw_1)|^2\,dx \leq \varepsilon.  
	\end{equation}
	Here, $c > 0$ depends on $n,l,L,i_g$ and $s_g$, but is independent of $\varepsilon$.
	
\end{lemma}

\begin{proof}
We first prove the first estimate in \eqref{vuw1}.
	Since $	u-h\in w_1+ W_0^{1,G}(\Omega_5)$ satisfies 
	\[
	\psi_1-h\leq u-h\leq\psi_2-h
	\quad\text{a.e. in }\Omega_5.
	\]
	Thus, $u-h$ is an comparison function in the variational inequality
	satisfied by $w_1$. It follows that
	\begin{align}\label{ineqw1}
		&\int_{\Omega_5}
		a(x,Dw_1)\cdot D(u-h-w_1)\,dx
		\geq
		\int_{\Omega_5}
		b(x,F)\cdot D(u-h-w_1)\,dx.
	\end{align}
By applying \eqref{u1}, \eqref{u2}, \eqref{a(x)4}, Lemma \ref{gyoung} 
 and Lemma \ref{vg}, we have
\begin{align*}
	\fint_{\Omega_5} G(|Dw_1|)\,dx
	&\leq c\fint_{\Omega_5}
	a(x,Dw_1)\cdot Dw_1\,dx
	\\
	&\leq	c\fint_{\Omega_5}	a(x,Dw_1)\cdot D(u-h)\,dx
+	c\fint_{\Omega_5}
	b(x,F)\cdot D(w_1+h-u)\,dx	\\
	&\leq	c\varepsilon_1	\fint_{\Omega_5}
	G^{*}\bigl(g(|Dw_1|)\bigr)\,dx+	c(\varepsilon_1)
	\fint_{\Omega_5}	G(|Du-Dh|)\,dx
	\\
	&\quad+
	c\varepsilon_2
	\fint_{\Omega_5}
	G(|Dw_1|)\,dx+	c(\varepsilon_2)
	\fint_{\Omega_5}
	G^{*}\bigl(g(|F|)\bigr)\,dx	\\
		&\leq	c(\varepsilon_1+\varepsilon_2)
	\fint_{\Omega_5}	G(|Dw_1|)\,dx	+
	c(\varepsilon_1,\varepsilon_2)
	\fint_{\Omega_5}
	\bigl[
	G(|F|)
	+
	G(|Du|)
	+
	G(|Dh|)
	\bigr]\,dx.
\end{align*}
	Choose $\varepsilon_1,\varepsilon_2>0$ sufficiently small so that
	$
	c(\varepsilon_1+\varepsilon_2)\leqslant \frac{1}{2}.
	$
	Absorbing the corresponding term into the left-hand side and using
	\eqref{ass2}, we obtain
	\[
	\fint_{\Omega_5} G(|Dw_1|)\,dx \leq c.
	\]

 	We then prove the second estimate in \eqref{vuw1}.	Since
 $	w_1+h\in u+W_0^{1,G}(\Omega_5)$ satisfies
 \[
 \psi_1\leq w_1+h\leq\psi_2
 \quad\text{a.e. in }\Omega_5.
 \]
 The function defined by
 \[
 v(x)=
 \begin{cases}
 	w_1(x)+h(x), & x\in\Omega_5,\\
 	u(x), & x\in\Omega\setminus\Omega_5,
 \end{cases}
 \]
 belongs to $\mathcal{A}_h(\Omega)$. Therefore, it is an admissible test
 function in the variational inequality \eqref{u} satisfied by $u$, and we obtain
 \begin{align*}
 	\int_{\Omega_5}
 	a(x,Du)\cdot D(w_1+h-u)\,dx
 	\geqslant
 	\int_{\Omega_5}
 	b(x,F)\cdot D(w_1+h-u)\,dx.
 \end{align*}
	Combining this inequality with \eqref{ineqw1},   we   obtain
	\begin{align*}
		&\int_{\Omega_5}
		\bigl[a(x,Du)-a(x,Dw_1)\bigr]
		\cdot(Du-Dw_1-Dh)\,dx
		\leq0.
	\end{align*}
	Subsequently, using \eqref{u1}, \eqref{u2}, \eqref{a(x)4}, Lemma \ref{gyoung} 
	and Lemma \ref{vg}, we have
	\begin{align*}
		\fint_{\Omega_5} |V(Du)-V(Dw_1)|^2\,dx &\leq c\fint_{\Omega_5}
		\bigl[a(x,Du)-a(x,Dw_1)\bigr]
		\cdot(Du-Dw_1)\,dx
		\\
		& \leq c\fint_{\Omega_5}
		\bigl[a(x,Du)-a(x,Dw_1)\bigr]
		\cdot Dh\,dx \\
			& \leq	c\varepsilon_3	\fint_{\Omega_5}
		G^{*}\bigl(g(|Du|)\bigr)\,dx+		c\varepsilon_4	\fint_{\Omega_5}	G^{*}\bigl(g(|Dw_1|)\bigr)\,dx \\
		&\quad +c(\varepsilon_3,\varepsilon_4)	\fint_{\Omega_5}G(|Dh|)dx \\
			& \leq c\varepsilon_3	\fint_{\Omega_5}
			G(|Du|)dx+c\varepsilon_4\fint_{\Omega_5}	G(|Dw_1|)dx+c(\varepsilon_3,\varepsilon_4)	\fint_{\Omega_5}G(|Dh|)dx.
	\end{align*}
	Given $\varepsilon>0$, we first choose $\varepsilon_3,\varepsilon_4$ small enough such that 
	\[ c\varepsilon_3	\fint_{\Omega_5}
	G(|Du|)dx+c\varepsilon_4\fint_{\Omega_5}	G(|Dw_1|)dx \leq \frac{\varepsilon}{2}.
	\]
	Then we take $\delta$ small enough such that 
	\[ c(\varepsilon_3,\varepsilon_4)	\fint_{\Omega_5}G(|Dh|)dx\leq \frac{\varepsilon}{2}.
	\]
Therefore, 
\[ \fint_{\Omega_5} |V(Du)-V(Dw_1)|^2\,dx \leq \varepsilon.
\]
 This completes the proof.
\end{proof}

\begin{lemma}\label{w1w2l}
	Assume that \eqref{u1}-\eqref{u3} are satisfied and   $w_1 \in u-h+ W_0^{1,G}(\Omega_5)$ with $\psi_1-h\leqslant w_1\leqslant \psi_2-h $ a.e. in $\Omega_5$ solves the inequality \eqref{w1}.  
	For any $\varepsilon > 0$ there exists a small $\delta=\delta(n,l,L,i_g,s_g,\varepsilon) \in (0,\frac{1}{8})$ such that if \eqref{ass2} holds for such $\delta$ and  $w_2 \in W^{1,G}(\Omega_5)$ with $ w_2 \geqslant \psi_1-h$ a.e. in $\Omega_5$ solves
	\begin{equation}\label{w2}
		\left\{\begin{array}{r@{\ \ }c@{\ \ }ll}
			&\int_{\Omega_5}  a(x,  Dw_2) \cdot D(v-w_2)dx \geqslant \int_{\Omega_5}  a(x,  D\psi_2-Dh) \cdot D(v-w_2)dx,  \\[0.05cm]
			&w_2=w_1 \ \ \mbox{on}\ \ \partial \Omega_5 \,,
		\end{array}\right.
	\end{equation}
	where $v\in w_2+W_0^{1,G}(\Omega_5)$ satisfies $ v \geqslant \psi_1-h$ a.e. in $\Omega_5$.
	Then
	\begin{equation}\label{w1w2}
		\fint_{\Omega_5} G(|Dw_2|)\,dx \leq c \quad \text{and} \quad \fint_{\Omega_5} |V(Dw_1)-V(Dw_2)|^2\,dx \leq \varepsilon.  
	\end{equation}
	Here, $c > 0$ depends on $n,l,L,i_g$ and $s_g$, but is independent of $\varepsilon$.
	\end{lemma}
	
	 \begin{proof}
	The first inequality in \eqref{w1w2} is proved as follows.
	 Taking $v=w_1\in w_2+W_0^{1,G}(\Omega_5)$ with
	 $v\geqslant\psi_1-h$ a.e.\ in $\Omega_5$
	 as test function in \eqref{w2}, we derive
	 \begin{align}\label{w22}
	 	 \int_{\Omega_5}a(x,Dw_2)\cdot D(w_1-w_2)\,dx
	 	\geqslant
	 	\int_{\Omega_5}a(x,D\psi_{2}-Dh)\cdot D(w_1-w_2)\,dx.
	 \end{align}
	 By applying \eqref{u1}, \eqref{a(x)4}, Lemma \ref{gyoung},
	 and Lemma \ref{vg}, we have
	 \begin{align*}
	 	\fint_{\Omega_5}G(|Dw_2|)\,dx
	 	&\leqslant  
	 	c\fint_{\Omega_5}a(x,Dw_2)\cdot Dw_2\,dx
	 	\\
	 	&\leqslant
	 	c\fint_{\Omega_5}a(x,Dw_2)\cdot Dw_1\,dx
	 	+
	 	c\fint_{\Omega_5}a(x,D\psi_{2}-Dh)\cdot D(w_2-w_1)\,dx
	 	\\
	 	&\leqslant
	 	c\varepsilon_1\fint_{\Omega_5}G^{*}\bigl(g(|Dw_2|)\bigr)\,dx
	 	+
	 	c(\varepsilon_1)\fint_{\Omega_5}G(|Dw_1|)\,dx
	 	\\
	 	&\quad+
	 	c\varepsilon_2\fint_{\Omega_5}G(|Dw_2|)dx
	 	+
	 	c(\varepsilon_2)\fint_{\Omega_5}G^*(g(|D\psi_{2}-Dh|))\,dx
	 	\\
	 	&\leqslant
	 	c(\varepsilon_1+\varepsilon_2)\fint_{\Omega_5}G(|Dw_2|)\,dx
	 	+
	 	c(\varepsilon_1,\varepsilon_2)
	 	\fint_{\Omega_5}\bigl[G(|Dw_1|)+G(|D\psi_{2}|)+G(|Dh|)\bigr]\,dx.
	 \end{align*}
	 We take $\varepsilon_1,\varepsilon_2$ small enough so that
	 $c(\varepsilon_1+\varepsilon_2)<1$.
	 Then, using \eqref{ass2} and \eqref{vuw1}, we obtain
	 \[
	 \fint_{\Omega_5}G(|Dw_2|)\,dx\leqslant c.
	 \]
	 
	 The second inequality in \eqref{w1w2} is proved below.
	 Using Lemma 3.1 in \cite{blo}, we have $w_2\leqslant \psi_2-h$ a.e. in $\Omega_5$. Then, taking $v=w_2$ in \eqref{w1}  and combining with \eqref{w22}, we obtain 
	 \begin{align*}
	 	\fint_{\Omega_5}
	 	\left|V(Dw_1)-V(Dw_2)\right|^2\,dx
	 	&\leq
	 	c\fint_{\Omega_5}
	 	\bigl[a(x,Dw_1)-a(x,Dw_2)\bigr]
	 	\cdot(Dw_1-Dw_2)\,dx
	 	\\
	 	&\leq
	 	c\fint_{\Omega_5}
	 	\left[
	 	b(x,F)-a\bigl(x,D\psi_2-Dh\bigr)
	 	\right]
	 	\cdot D(w_1-w_2)\,dx
	 	\\
	 	&\leq
	 	c(\varepsilon_3+\varepsilon_4)
	 	\fint_{\Omega_5}
	 	G(|Dw_1-Dw_2|)\,dx
	 	\\
	 	&\quad+
	 	c(\varepsilon_3,\varepsilon_4)
	 	\fint_{\Omega_5}
	 	\left[
	 	G^*\bigl(g(|F|)\bigr)
	 	+
	 	G^*\bigl(g(|D\psi_2-Dh|)\bigr)
	 	\right]\,dx
	 	\\
	 	&\leq
	 	c(\varepsilon_3+\varepsilon_4)+
	 	c(\varepsilon_3,\varepsilon_4)
	 	\fint_{\Omega_5}
	 	\left[
	 	G(|F|)
	 	+
	 	G(|D\psi_2|)
	 	+
	 	G(|Dh|)
	 	\right]\,dx.
	 \end{align*}
	 We choose $\varepsilon_3,\varepsilon_4>0$ sufficiently small such that
	 \[
	 c(\varepsilon_3+\varepsilon_4)\leq\frac{\varepsilon}{2}.
	 \]
	 After fixing $\varepsilon_3$ and $\varepsilon_4$, we choose $\delta>0$
	 sufficiently small such that
	 \[
	 c(\varepsilon_3,\varepsilon_4)\delta
	 \leq\frac{\varepsilon}{2}.
	 \]
	 This completes the proof.
	\end{proof}
	
		\begin{lemma}\label{w2w3l}
		Assume that \eqref{u1}-\eqref{u3} are satisfied and   $w_2 \in w_1+ W_0^{1,G}(\Omega_5)$ with $ w_2 \geqslant \psi_1-h$ a.e. in $\Omega_5$ solves  the inequality \eqref{w2}.  
		For any $\varepsilon > 0$ there exists a small $\delta=\delta(n,l,L,i_g,s_g,\varepsilon) \in (0,\frac{1}{8})$ such that if \eqref{ass2} holds for such $\delta$ and  $w_3 \in W^{1,G}(\Omega_5)$ is a weak solution of 
		\begin{equation}\label{w3}
			\left\{\begin{array}{r@{\ \ }c@{\ \ }ll}
					-\operatorname{div}a(x,Dw_3) &=&-\operatorname{div}a(x,D\psi_1-Dh) \ \ &\mbox{in}\ \   \Omega_5 \,,  \\[0.05cm]
				 w_3 &=&w_2 \ \ &\mbox{on}\ \ \partial \Omega_5 \,.
			\end{array}\right.
		\end{equation}
		Then
		\begin{equation}\label{w2w3}
			\fint_{\Omega_5} G(|Dw_3|)\,dx \leq c \quad \text{and} \quad \fint_{\Omega_5} |V(Dw_2)-V(Dw_3)|^2\,dx \leq \varepsilon.  
		\end{equation}
		Here, $c > 0$ depends on $n,l,L,i_g$ and $s_g$, but is independent of $\varepsilon$.
	\end{lemma}
	
	\begin{proof}
	We take $w_3 - w_2 \in W_0^{1,G}(\Omega_5)$ as a test function in \eqref{w3} to get
	\begin{align}\label{aw3}
		\int_{\Omega_5} a(x, Dw_3) \cdot D(w_3 - w_2) dx &= \int_{\Omega_5} a(x, D\psi_1 - Dh) \cdot D(w_3 - w_2) dx 
	\end{align}
	Using \eqref{u1}, \eqref{a(x)4}, Lemma \ref{gyoung} and Lemma \ref{vg} to obtain
	\begin{align*}
		\fint_{\Omega_5} G(|Dw_3|) dx &\le c \fint_{\Omega_5} a(x, Dw_3) \cdot Dw_3 dx \\
		&\le c \fint_{\Omega_5} a(x, Dw_3) \cdot Dw_2 dx + c \fint_{\Omega_5} a(x, D\psi_1 - Dh) \cdot (Dw_3 - Dw_2) dx \\
		&\le c(\varepsilon_1 + \varepsilon_2) \fint_{\Omega_5} G(|Dw_3|) dx + c(\varepsilon_1, \varepsilon_2) \fint_{\Omega_5}\left[  G(|Dw_2|) + G(|D\psi_1|) + G(|Dh|)\right]  dx
	\end{align*}
	Now we choose $\varepsilon_1, \varepsilon_2$ small such that $c(\varepsilon_1 + \varepsilon_2) \le \frac{1}{2}$.
	Then by \eqref{ass2} and \eqref{w1w2}, we obtain the first estimate in \eqref{w2w3}.
	
	Next, we prove the second estimate in \eqref{w2w3}.
	Because of Lemma  3.2  in \cite{blo}, $w_3 \ge \psi_1 - h$ a.e. in $\Omega_5$.
	So taking $v = w_3$ in \eqref{w2}, combining with \eqref{aw3} to obtain
	\begin{align*}
		\fint_{\Omega_5} |V(Dw_2) - V(Dw_3)|^2 dx &\le c \fint_{\Omega_5} \left[ a(x, Dw_2) - a(x, Dw_3)\right]  \cdot (Dw_2 - Dw_3) dx \\
		&\le c \fint_{\Omega_5} [a(x, D\psi_2 - Dh) - a(x, D\psi_1 - Dh)] \cdot D(w_2 - w_3) dx \\
		&\le c(\varepsilon_3, \varepsilon_4) \fint_{\Omega_5} [G(|D\psi_1|) + G(|D\psi_2|) + G(|Dh|)] dx \\
		&\quad + c(\varepsilon_3 + \varepsilon_4) \fint_{\Omega_5} [G(|Dw_2|) + G(|Dw_3|)] dx
	\end{align*}
	we first choose $\varepsilon_3, \varepsilon_4$ small enough such that
	\begin{align*}
		c(\varepsilon_3 + \varepsilon_4) \fint_{\Omega_5} \left[ G(|Dw_2|) + G(|Dw_3|)\right]  dx \le \frac{\varepsilon}{2}
	\end{align*}
	Then we choose $\delta$ small enough so that
	\begin{align*}
		c(\varepsilon_3, \varepsilon_4) \fint_{\Omega_5} [G(|D\psi_1|) + G(|D\psi_2|) + G(|Dh|)] dx \le \frac{\varepsilon}{2}.
	\end{align*}
		Therefore, this completes the proof.
	
	\end{proof}
	
The remaining comparison estimates have been proved in \cite{cho1}, and we present the results for completeness. We first consider the interior case for $B_5\subseteq \Omega$.

	\begin{lemma}\label{w3w4l}
	Assume that \eqref{u1}-\eqref{u3} are satisfied and   $w_3 \in w_2+ W_0^{1,G}(\Omega_5)$  is a weak solution of \eqref{w3}. 
	For any $\varepsilon > 0$ there exists a small $\delta=\delta(n,l,L,i_g,s_g,\varepsilon) \in (0,\frac{1}{8})$ such that if \eqref{ass1}  and \eqref{ass2}   holds for such $\delta$ and  $w_4 \in W^{1,G}(B_4)$ is a weak solution of 
	\begin{equation*} 
		\left\{\begin{array}{r@{\ \ }c@{\ \ }ll}
			-\operatorname{div}\bar{a}_{B_4}(Dw_4) &=&0 \ \ &\mbox{in}\ \   B_4 \,,  \\[0.05cm]
			w_4 &=&w_3 \ \ &\mbox{on}\ \ \partial B_4 \,.
		\end{array}\right.
	\end{equation*}
	Then
	\begin{equation*} 
		\|G(|Dw_4|)\|_{L^{\infty}(B_2)} \leq c \quad \text{and} \quad \fint_{B_4} |V(Dw_3)-V(Dw_4)|^2\,dx \leq \varepsilon.  
	\end{equation*}
	Here, $c > 0$ depends on $n,l,L,i_g$ and $s_g$, but is independent of $\varepsilon$.
\end{lemma}

Next, we consider the boundary case for $B_5\nsubseteq \Omega$. 

		\begin{lemma}\label{w3w4b}
		Assume that \eqref{u1}-\eqref{u3} are satisfied and   $w_3 \in w_2+ W_0^{1,G}(\Omega_5)$  is a weak solution of \eqref{w3}. 
		For any $\varepsilon > 0$ there exists a small $\delta=\delta(n,l,L,i_g,s_g,\varepsilon) \in (0,\frac{1}{8})$ such that if \eqref{ass1},  \eqref{ass2} and \eqref{ass3} holds for such $\delta$ and  $w_4 \in W^{1,G}(B_4^+)$ is a weak solution of 
		\begin{equation*} 
			\left\{\begin{array}{r@{\ \ }c@{\ \ }ll}
				-\operatorname{div}\bar{a}_{B_4^+}(Dw_4) &=&0 \ \ &\mbox{in}\ \   B_4^+ \,,  \\[0.05cm]
				w_4 &=&0 \ \ &\mbox{on}\ \   T_4 \,.
			\end{array}\right.
		\end{equation*}
		Then
		\begin{equation*} 
			\|G(|D\bar{w}_4|)\|_{L^{\infty}(\Omega_2)} \leq c \quad \text{and} \quad \fint_{\Omega_3} |V(Dw_3)-V(D\bar{w}_4)|^2\,dx \leq \varepsilon,  
		\end{equation*}
	where $\bar{w}_4$ is the zero extension of $w_4$ from $B_4^+$ to $B_4$ and $c > 0$ depends on $n,l,L,i_g$ and $s_g$, but is independent of $\varepsilon$.
	\end{lemma}

\section{Proof of main theorem } 
 
 In this section, we establish the global weighted Orlicz estimate using the comparison estimates in the previous section and a Vitali type covering lemma.
 
 \begin{lemma}\label{lem5.1}
 	There exists a constant $C = C(n,l,L,i_g, s_g ) > 1$ such that for any $\varepsilon > 0$, there exists $\delta = \delta(\varepsilon,n,l,L,i_g, s_g) \in (0, \frac{1}{8})$ ensuring that if $(a, \Omega)$ satisfies the $(\delta, R)$-vanishing condition, then the following holds. For any $y \in \Omega$, $r \in (0, \frac{R}{40})$, and $\lambda \geq 1$, provided that the set 
 	\begin{align}\label{ml}
 		&\left\{
 		\begin{aligned}
 			x\in\Omega_r(y):\quad
 			&\mathcal{M}\bigl(G(|Du|)\bigr)(x)\leq\lambda,\quad
 			\mathcal{M}\bigl(G(|F|)\bigr)(x)\leq\frac{\lambda\delta}{4},\quad
 			\mathcal{M}\bigl(G(|D\psi_1|)\bigr)(x)\leq\frac{\lambda\delta}{4},
 			\\
 			&\mathcal{M}\bigl(G(|D\psi_2|)\bigr)(x)\leq\frac{\lambda\delta}{4},\quad
 			\mathcal{M}\bigl(G(|Dh|)\bigr)(x)\leq\frac{\lambda\delta}{4}
 		\end{aligned}
 		\right\}
 		\neq\varnothing,
 	\end{align}
 	then
 	\begin{align}
 		|\{x \in \Omega_r(y) : \mathcal{M}(G(|Du|)) > \lambda C\}| < \varepsilon |B_r(y)|.
 	\end{align}
 \end{lemma}
 
 \begin{proof}
 	By our assumption in \eqref{ml}, there exists a point $y_1 \in \Omega_r(y)$ such that for every $\rho > 0$, we have
  \begin{equation}\label{ml2}
  	\left\{
  	\begin{aligned}
  		\frac{1}{|B_\rho(y_1)|}
  		\int_{\Omega_\rho(y_1)}G(|Du|)\,dx
  		&\leq \lambda, \\[0.15cm]
  		\frac{1}{|B_\rho(y_1)|}
  		\int_{\Omega_\rho(y_1)}G(|F|)\,dx
  		&\leq \frac{\lambda\delta}{4}, \\[0.15cm]
  		\frac{1}{|B_\rho(y_1)|}
  		\int_{\Omega_\rho(y_1)}G(|D\psi_i|)\,dx
  		&\leq \frac{\lambda\delta}{4},	\qquad i=1,2,  \\[0.15cm]
  		\frac{1}{|B_\rho(y_1)|}
  		\int_{\Omega_\rho(y_1)}G(|Dh|)\,dx
  		&\leq \frac{\lambda\delta}{4}.
  	\end{aligned}
  	\right.
  \end{equation}
 	Let us first consider the case where $B_{5r}(y) \subset \Omega$. This geometric condition ensures the inclusion $B_{5r}(y) \subset \Omega_{6r}(y_1)$. Combining this fact with the estimates in \eqref{ml2} yields
 	\begin{equation}\label{ml3}
 		\left\{
 		\begin{aligned}
 				\fint_{B_{5r}(y)} G(|Du|) \, dx  &\leq 2^n \lambda,  \\[0.15cm]
 		 \fint_{B_{5r}(y)} G(|F|) \, dx &\leq \frac{ 2^n \lambda \delta}{4}, \\[0.15cm]
 		  \fint_{B_{5r}(y)} G(|D\psi_i|) \, dx &\leq \frac{ 2^n \lambda \delta}{4},	\qquad i=1,2,  \\[0.15cm]
 		\fint_{B_{5r}(y)} G(|Dh|) \, dx &\leq \frac{ 2^n \lambda \delta}{4}.
 		\end{aligned}
 		\right.
 	\end{equation}
 	To proceed, we introduce the following rescaling argument.
 		\begin{equation*} 
 		\left\{
 		\begin{aligned}
 			 	&\tilde{u}(x) = \frac{u(y + rx)}{2^n \lambda r}, \quad \tilde{F}(x) = \frac{F(y + rx)}{2^n \lambda},  \quad \tilde{\psi_i}(x) = \frac{\psi_i(y + rx)}{2^n \lambda r}, \quad  i=1,2, \\[0.15cm]
 			 	&\tilde{a}(x,\eta) = a(y + rx,2^n \lambda \eta), \quad 	\tilde{b}(x,\eta) = b(y + rx,2^n \lambda \eta), \quad \tilde{h}(x) = \frac{h(y + rx)}{2^n \lambda r},   \\[0.15cm]
 		&\tilde{g}(t) = g(2^n \lambda t), \quad \tilde{G}(t) =\int_0^t \tilde{g}(s) ds, \quad t \geq 0, \\[0.15cm]
 		\end{aligned}
 		\right.
 	\end{equation*}
 	where
 	\[
 	x \in \tilde{\Omega} = \left\{ \frac{z - y}{r} : z \in \Omega \right\}. 
 	\]
 	Then $\tilde{a}, \tilde{b}$ and $\tilde{g}$ satisfy \eqref{u1}, \eqref{u2}, \eqref{u3}, and $\tilde{u}\in \tilde{\mathcal{A}}_{\tilde{h}}(\tilde{\Omega})$ solves the variational inequality
 	\begin{equation*} 
 		\int_{\tilde{\Omega}}  \tilde{a}(x, D \tilde{u}) \cdot D(v-\tilde{u}) \mathrm{d} x \geqslant \int_{\tilde{\Omega}} \tilde{b}(x, \tilde{F} ) \cdot D(v-\tilde{u}) \mathrm{d}x,  
 	\end{equation*}
 	where 
 	$$
 	\tilde{\mathcal{A}}_{\tilde{h}}(\tilde{\Omega})=\left\{\varphi \in W_{\tilde{h}}^{1, \tilde{G}}(\tilde{\Omega}): \tilde{\psi}_{1} \leq \varphi \leq \tilde{\psi}_{2} \text { a.e. in } \tilde{\Omega} \right\} .
 	$$
 	   Therefore, conditions \eqref{ass1} and \eqref{ass2} are satisfied, and $\frac{R}{r}>5$. By applying Lemmas \ref{uw1}, \ref{w1w2l}, \ref{w2w3l}, and \ref{w3w4l} and scaling back, we obtain a function $w_4 \in W^{1,G}(B_{4r}(y))$ such that
 	\begin{align}\label{ml4}
 			\fint_{B_{4r}(y)} |V(Du) - V(Dw_4)|^2 \, dx < 2^n \lambda \varepsilon_1 \quad \text{and} \quad \|G(|Dw_4|)\|_{L^\infty(B_{2r}(y))} \leq 2^n \lambda C_1, 
 	\end{align}
 	where $\varepsilon_1>0$ is to be chosen later, and $C_1>1$ is a constant depending only on $n,l,L,i_g,s_g$.
 	Set
 	\[
 	C\geq \max\{2^{n+1}c_vC_1,3^n\}
 	\]
 	for $c_v$ as in \eqref{cv}. We next claim the following inclusion:
 	\begin{align}\label{da}
 		\left\{x\in B_r(y):\mathcal{M}(G(|Du|))>\lambda C\right\}
 		\subset
 		\left\{x\in B_r(y):
 		\mathcal{M}_{B_{4r}(y)}
 		\left(|V(Du)-V(Dw_4)|^2\right)>2^n\lambda
 		\right\}.
 	\end{align}
 To this end, let us take a point
 	\[
 	x_1\in
 	\left\{
 	x\in B_r(y):
 	\mathcal{M}_{B_{4r}(y)}
 	\left(|V(Du)-V(Dw_4)|^2\right)
 	\leq 2^n\lambda
 	\right\}.
 	\]
 First, we consider the case where $0<\rho<r$. Then $B_\rho(x_1)\subset B_{2r}(y)$, and combining \eqref{cv} with \eqref{ml4} yields
 	\begin{align*}
 		\frac{1}{|B_\rho(x_1)|}
 		\int_{\Omega_\rho(x_1)}G(|Du|)\,dx
 		&\leq
 		c_v
 		\fint_{B_\rho(x_1)}
 		\left(
 		|V(Du)-V(Dw_4)|^2
 		+
 		G(|Dw_4|)
 		\right)dx
 		\nonumber\\
 		&\leq
 		c_v(2^n\lambda+2^n\lambda C_1)
 		\leq \lambda C.
 	\end{align*}
  Next, we consider the case where $\rho\geq r$. In this situation, the geometric inclusion $B_\rho(x_1)\subset B_{3\rho}(y_1)$ combined with \eqref{ml2} implies that
 	\begin{align*}
 		\frac{1}{|B_\rho(x_1)|}
 		\int_{\Omega_\rho(x_1)}G(|Du|)\,dx
 		&\leq
 		\frac{|B_{3\rho}(y_1)|}{|B_\rho(x_1)|}
 		\frac{1}{|B_{3\rho}(y_1)|}
 		\int_{\Omega_{3\rho}(y_1)}G(|Du|)\,dx
 		\nonumber\\
 		&\leq 3^n\lambda
 		\leq \lambda C .
 	\end{align*}
 	 Consequently, we obtain
 	\begin{align*}
 		x_1\in
 		\left\{
 		x\in B_r(y):
 		\mathcal{M}(G(|Du|))\leq \lambda C
 		\right\},
 	\end{align*}
which establishes the claimed inclusion \eqref{da}.
 	
 Now, by virtue of \eqref{one}, \eqref{ml4}, and \eqref{da}, we can estimate
 	\begin{align*}
 		\left|
 		\left\{
 		x\in B_r(y):
 		\mathcal{M}(G(|Du|))>\lambda C
 		\right\}
 		\right|
 		&\leq
 		\left|
 		\left\{
 		x\in B_r(y):
 		\mathcal{M}_{B_{3r}(y)}
 		(|V(Du)-V(Dw_4)|^2)>2^n\lambda
 		\right\}
 		\right|
 		\nonumber\\
 		&\leq
 		\frac{c|B_r(y)|}{2^n\lambda}
 		\fint_{B_{3r}(y)}
 		|V(Du)-V(Dw_4)|^2\,dx
 		\nonumber\\
 		&<
 		c\varepsilon_1|B_r(y)|.
 	\end{align*}
 Finally, the proof is completed by choosing $\varepsilon_1$ sufficiently small such that $c\varepsilon_1\leq \varepsilon$.

 	Next, we address the boundary case where $B_{5r}(y) \nsubseteq \Omega$, meaning there exists a point $\bar{x} \in B_{5r}(y) \cap \partial\Omega$. In light of Remark \ref{rei} and the assumption $0 < 40r < R$, we can select a local coordinate system $\{\xi_1, \ldots, \xi_n\}$ such that $\bar{x} = -40r\delta e_n$ and the following geometric inclusion holds:
 	\begin{align}\label{rei2}
 		B^+_{35r}\subset \Omega_{35r}
 		\subset
 		\{z\in B_{35r}:z_n>-80r\delta\}.
 	\end{align}
 	Within this coordinate system, we have $|y| \leq 10r$ and $|y_1| \leq 11r$. This implies the inclusions $\Omega_r(y) \subset \Omega_{11r}$ and $\Omega_{35r} \subset \Omega_{46r}(y_1)$. Consequently, it follows from \eqref{ml2} that
 		\begin{equation}\label{rei3}
 		\left\{
 		\begin{aligned}
 			 	\fint_{\Omega_{35r}}G(|Du|)\,dz
 			 &\leq 2^{n+1}\lambda, \\[0.15cm]
 		 \fint_{\Omega_{35r}}G(|F|)\,dz
 		 &\leq \frac{2^{n+1}\lambda\delta}{4}, \\[0.15cm]
 		 \fint_{\Omega_{35r}}G(|Dh|)\,dz
 		 &\leq \frac{2^{n+1}\lambda\delta}{4}, \\[0.15cm]
 		 \fint_{\Omega_{35r}}G(|D\psi_i|)\,dz
 		 &\leq \frac{2^{n+1}\lambda\delta}{4}, \quad i=1,2.
 		\end{aligned}
 		\right.
 	\end{equation}
 	To proceed, we introduce the following standard rescaling argument:
 	\begin{equation*} 
 		\left\{
 		\begin{aligned}
 			&\tilde{u}(x) = \frac{u(7rx)}{2^{n+1} \lambda 7r}, \quad \tilde{F}(x) = \frac{F( 7rx)}{2^{n+1} \lambda r},  \quad \tilde{\psi_i}(x) = \frac{\psi_i(7 rx)}{2^{n+1} \lambda 7r}, \quad  i=1,2, \\[0.15cm]
 			&\tilde{a}(x,\eta) = a(7 rx,2^{n+1} \lambda \eta), \quad 	\tilde{b}(x,\eta) = b( 7rx,2^{n+1}\lambda \eta), \quad \tilde{h}(x) = \frac{h(7 rx)}{2^{n+1} \lambda 7r},   \\[0.15cm]
 			&\tilde{g}(t) = g(2^{n+1}\lambda t), \quad \tilde{G}(t) =\int_0^t \tilde{g}(s) ds, \quad t \geq 0.  \\[0.15cm]
 		\end{aligned}
 		\right.
 	\end{equation*} 
 	Under this scaling, conditions \eqref{ass1} and \eqref{ass2} remain satisfied, and we clearly have $\frac{R}{7r} > 5$.
 	Proceeding analogously to the interior case, we apply Lemmas \ref{uw1}, \ref{w1w2l}, \ref{w2w3l}, and \ref{w3w4b}. Scaling back the variables, we deduce the existence of a reference function $w_4 \in W^{1,G}(\Omega_{28r})$ satisfying
 	\begin{align*}
 		\fint_{\Omega_{21r}}
 		|V(Du)-V(Dw_4)|^2\,dz
 		 \leq 2^{n+1}\lambda \varepsilon_2 \quad \& \quad  \|G(|Dw_4|)\|_{L^\infty(\Omega_{14r})}
 		 \leq
 		2^{n+1}\lambda C_2 ,
 	\end{align*}
 where $\varepsilon_2$ is a small parameter to be determined later, and $C_2 > 1$ is a constant depending only on $n, l, L, i_g, s_g$.
 By further enforcing the condition
 	\[
 C\geq \max\{2^{n+2}c_vC_2,27^n\},
 	\]
 	we can argue exactly as in the case $B_{5r}(y) \subset \Omega$ to deduce the set inclusion
 	\begin{align*}
 		\left\{
 		z\in\Omega_{11r}:
 		\mathcal{M}(G(|Du|))>\lambda C
 		\right\}
 		\subset
 		\left\{
 		z\in\Omega_{11r}:
 		\mathcal{M}_{\Omega_{21r}}
 		\left(
 		|V(Du)-V(Dw_4)|^2
 		\right)
 		>2^{n+1}\lambda
 		\right\},
 	\end{align*}
This inclusion naturally yields the measure estimate
 	\begin{align*}
 		\left|
 		\left\{
 		z\in\Omega_{11r}:
 		\mathcal{M}(G(|Du|))>\lambda C
 		\right\}
 		\right|
 		<
 		c\varepsilon_2|B_r(y)|.
 	\end{align*}
 	Finally, observing that $\Omega_r(y) \subset \Omega_{11r}$, we arrive at the desired local bound
 	\begin{align*}
 		\left|
 		\left\{
 		z\in\Omega_r(y):
 		\mathcal{M}(G(|Du|))>\lambda C
 		\right\}
 		\right|
 		<
 		c\varepsilon_2|B_r(y)|.
 	\end{align*}
 The proof is then completed by choosing $\varepsilon_2$ sufficiently small such that $c\varepsilon_2 \leq \varepsilon$.
 \end{proof}

 \begin{lemma}\label{cor5.2}
 	Let $\omega$ be a weight in $A_q$ for $1<q<\infty$, and let $C>1$ be the constant given in Lemma \ref{lem5.1}. Then, for any $\varepsilon>0$, there exists 
 	$
 	\delta=\delta(\varepsilon,n,l,L,i_g,s_g,\omega)
 	\in\left(0,\frac18\right)
 	$ 
 	such that if $(a,\Omega)$ satisfies the $(\delta,R)$-vanishing condition, the following property holds. For any $y\in\Omega$, $r\in(0,\frac{R}{40})$, and $\lambda\geq1$, provided that the condition
 	\begin{align}\label{5.10}
 		\omega
 		\left(
 		\left\{
 		x\in\Omega_r(y):
 		\mathcal{M}(G(|Du|))>\lambda C
 		\right\}
 		\right)
 		\geq
 		\varepsilon\omega(B_r(y))
 	\end{align}
 	is satisfied, we have  
 	\begin{align}\label{5.11}
 		\Omega_r(y)
 		\subset&
 		\left\{
 		x\in\Omega:
 		\mathcal{M}(G(|Du|))>\lambda
 		\right\}
 		\cup
 		\left\{
 		x\in\Omega:
 		\mathcal{M}(G(|F|))>\frac{\lambda\delta}{4}
 		\right\}
 		\nonumber\\
 		&\cup
 		\left\{
 		x\in\Omega:
 		\mathcal{M}(G(|D\psi_1|))>\frac{\lambda\delta}{4}
 		\right\} 	 \cup
 		\left\{
 		x\in\Omega:
 		\mathcal{M}(G(|D\psi_2|))>\frac{\lambda\delta}{4}
 		\right\} 	\nonumber\\
 		&\cup
 		\left\{
 		x\in\Omega:
 		\mathcal{M}(G(|Dh|))>\frac{\lambda\delta}{4}
 		\right\}. 
 	\end{align}
 \end{lemma}

 \begin{proof}
 	We argue by contradiction. Suppose that there exists an $\varepsilon_0 > 0$ such that for any $\delta > 0$, we can find a weak solution $u \in \mathcal{A}_h(\Omega)$ to the variational inequality \eqref{u} with $(a, \Omega)$ satisfying the $(\delta, R)$-vanishing condition, possessing the following property: for some $y \in \Omega$, $r \in (0, \frac{R}{40})$, and $\lambda > 1$, condition \eqref{5.10} holds with $\varepsilon = \varepsilon_0$, but the inclusion \eqref{5.11} is false.

 Consequently, we can assume that there exist $y_0 \in \Omega$, $r_0 \in (0, \frac{R}{40})$, and $\lambda_0 \geq 1$ such that
 	\begin{align}\label{5.12}
 		\omega
 		\left(
 		\left\{
 		x\in\Omega_{r_0}(y_0):
 		\mathcal{M}(G(|Du|))>\lambda_0C
 		\right\}
 		\right)
 		\geq
 		\varepsilon_0\omega(B_{r_0}(y_0))
 	\end{align}
 and the associated level set is non-empty, i.e.,
 \begin{align*}
 	&\left\{
 	\begin{aligned}
 		x\in\Omega_{r_0}(y_0):\quad
 		&\mathcal{M}\bigl(G(|Du|)\bigr)(x)\leq\lambda_0,\quad
 		\mathcal{M}\bigl(G(|F|)\bigr)(x)
 		\leq\frac{\lambda_0\delta}{4},
 		\\
 		&\mathcal{M}\bigl(G(|D\psi_1|)\bigr)(x)
 		\leq\frac{\lambda_0\delta}{4},\quad
 		\mathcal{M}\bigl(G(|D\psi_2|)\bigr)(x)
 		\leq\frac{\lambda_0\delta}{4},
 		\\
 		&\mathcal{M}\bigl(G(|Dh|)\bigr)(x)
 		\leq\frac{\lambda_0\delta}{4}
 	\end{aligned}
 	\right\}
 	\neq\varnothing .
 \end{align*}
 Based on the selection of $C$ and the application of Lemma \ref{lem5.1}, we can choose $\delta$ small enough to guarantee that
 	\begin{align*}
 		\left|
 		\left\{
 		x\in\Omega_{r_0}(y_0):
 		\mathcal{M}(G(|Du|))>\lambda_0C
 		\right\}
 		\right|
 		<
 		\left(\frac{\epsilon_0}{d}\right)^{\frac1\nu}
 		|B_{r_0}(y_0)|,
 	\end{align*}
 	where $d$ and $\nu$ are the structural constants associated with the weight $\omega$ as given in Lemma \ref{weight2}.
 	Then applying Lemma \ref{weight2} to obtain
 	\begin{align*}
 		&\omega
 		\left(
 		\left\{
 		x\in\Omega_{r_0}(y_0):
 		\mathcal{M}(G(|Du|))>\lambda_0C
 		\right\}
 		\right)
 		\nonumber\\
 		&\leq
 		d
 		\left(
 		\frac{
 			\left|
 			\left\{
 			x\in\Omega_{r_0}(y_0):
 			\mathcal{M}(G(|Du|))>\lambda_0C
 			\right\}
 			\right|
 		}
 		{|B_{r_0}(y_0)|}
 		\right)^\nu
 		\omega(B_{r_0}(y_0))
 		\nonumber\\
 		&<
 		\varepsilon_0\omega(B_{r_0}(y_0)).
 	\end{align*}
 This strictly contradicts our initial assumption \eqref{5.12}, thereby completing the proof.
 \end{proof}

\begin{lemma}\label{deome}
 Let $\omega \in A_q$ with $1 < q < \infty$. For any given $\varepsilon > 0$, there exists a constant 
 $\delta = \delta(n,l,L,i_g,s_g,\varepsilon, \omega) \in \left(0, \frac{1}{8}\right)$ 
 such that the following property holds.
 
 Suppose that for a collection of points $\{z_j\}_{j=1}^l \subset \overline{\Omega}$ and a radius $0 < r \leq \frac{R}{2000}$, the covering condition
 \[
 \overline{\Omega} \subset \bigcup_{j=1}^l B_r(z_j)
 \]
 is satisfied, alongside the density estimate
 \begin{align}\label{deome1}
 	\omega(\{x \in \Omega : \mathcal{M}(G(|Du|))(x) > C\}) < \varepsilon \omega(B_r(z_j)) \quad \text{for each }
 	j\in\{1, \dots, l\}. 
 \end{align}
 Then, for every integer $k \in \mathbb{N}$, we have the bound
 \begin{align}\label{deome2}
 	\omega(\{x \in \Omega : \mathcal{M}(G(|Du|))(x) > C^k\})
 	&\leq \varepsilon_1^k \omega(\{x \in \Omega : \mathcal{M}(G(|Du|))(x) > 1\}) \nonumber\\ 
 	&\quad + \sum_{i=1}^k \varepsilon_1^i \omega(\{x \in \Omega : \mathcal{M}(G(|F|))(x) > \frac{C^{k-i}\delta}{4}\}) \nonumber\\  
 	&\quad +\sum_{i=1}^k \varepsilon_1^i \omega(\{x \in \Omega : \mathcal{M}(G(|D\psi_1|))(x) > \frac{C^{k-i}\delta}{4}\}) \nonumber\\ 
 	&\quad +\sum_{i=1}^k \varepsilon_1^i \omega(\{x \in \Omega : \mathcal{M}(G(|D\psi_2|))(x) > \frac{C^{k-i}\delta}{4}\}) \nonumber\\ 
 	&\quad +\sum_{i=1}^k \varepsilon_1^i \omega(\{x \in \Omega : \mathcal{M}(G(|Dh|))(x) > \frac{C^{k-i}\delta}{4}\}),
 \end{align}
 where $C=C(n,l,L,i_g,s_g) > 1$ is the constant given in Lemma \ref{lem5.1}, and $\varepsilon_1 = 20^{nq}[\omega]_{A_q}^2\varepsilon$.
\end{lemma}

 \begin{proof}
 For a fixed $\varepsilon>0$, we begin by choosing $\delta$ according to Lemma \ref{cor5.2}. Next, for each $k\in\mathbb{N}$, we introduce the level sets
 	\begin{align*}
 		C_k
 		:=
 		\left\{
 		x\in\Omega:
 		\mathcal{M}(G(|Du|))(x)>C^k
 		\right\}
 	\end{align*}
 	and
 	\begin{align*}
 		D_k
 		:=&
 		\left\{
 		x\in\Omega:
 		\mathcal{M}(G(|Du|))(x)>C^{k-1}
 		\right\}
 		\nonumber\\
 		&\cup
 		\left\{
 		x\in\Omega:
 		\mathcal{M}(G(|F|))(x)>\frac{C^{k-1}\delta}{4}
 		\right\}
 		\nonumber\\
 		&\cup
 		\left\{
 		x\in\Omega:
 		\mathcal{M}(G(|D\psi_1|))(x)>\frac{C^{k-1}\delta}{4}
 		\right\} 
 			\nonumber\\
 		&\cup
 		\left\{
 		x\in\Omega:
 		\mathcal{M}(G(|D\psi_2|))(x)>\frac{C^{k-1}\delta}{4}
 		\right\}
 			\nonumber\\
 		&\cup
 		\left\{
 		x\in\Omega:
 		\mathcal{M}(G(|Dh|))(x)>\frac{C^{k-1}\delta}{4}
 		\right\}
 	\end{align*}
 	It is evident that $C_k\subset C_1$ for all $k\geq1$. Combining this inclusion with \eqref{deome1} yields
 	\begin{align*}
 		\omega(C_k)
 		\leq
 		\omega(C_1)
 		<
 		\varepsilon\omega(B_r(z_j)),
 		\qquad
 		\text{for all }j=1,\ldots,l .
 	\end{align*}

 	Furthermore, invoking Lemma \ref{cor5.2} with $\lambda=C^{k-1}$ ensures that all the hypotheses of Lemma \ref{weight5} are fulfilled. Consequently, we deduce that
 	\begin{align*}
 		\omega(C_k)
 		\leq
 		\varepsilon_1\omega(D_k),
 		\qquad
 		\text{for any }k\geq1 .
 	\end{align*}
Expanding this inequality explicitly, we obtain
 	\begin{align*}
 		&\omega
 		\left(
 		\left\{
 		x\in\Omega:
 		\mathcal{M}(G(|Du|))(x)>C^k
 		\right\}
 		\right)
 		\nonumber\\
 		&\leq
 		\varepsilon_1
 		\omega
 		\left(
 		\left\{
 		x\in\Omega:
 		\mathcal{M}(G(|Du|))(x)>C^{k-1}
 		\right\}
 		\right)
 		\nonumber\\
 		&\quad+
 		\varepsilon_1
 		\omega
 		\left(
 		\left\{
 		x\in\Omega:
 		\mathcal{M}(G(|F|))(x)>\frac{C^{k-1}\delta}{4}
 		\right\}
 		\right)
 		\nonumber\\
 		&\quad+
 		\varepsilon_1
 		\omega
 		\left(
 		\left\{
 		x\in\Omega:
 		\mathcal{M}(G(|D\psi_1|))(x)>\frac{C^{k-1}\delta}{4}
 		\right\}
 		\right)
 		\nonumber\\
 		&\quad+
 		\varepsilon_1
 		\omega
 		\left(
 		\left\{
 		x\in\Omega:
 		\mathcal{M}(G(|D\psi_2|))(x)>\frac{C^{k-1}\delta}{4}
 		\right\}
 		\right)
 		\nonumber\\
 		&\quad+
 		\varepsilon_1
 		\omega
 		\left(
 		\left\{
 		x\in\Omega:
 		\mathcal{M}(G(|Dh|))(x)>\frac{C^{k-1}\delta}{4}
 		\right\}
 		\right)
 	\end{align*}
for any $k\geq1$. Finally, iterating this relation inductively yields the desired estimate \eqref{deome2}.
 \end{proof}

\begin{proof}[Proof of Theorem \ref{weight}]
 From Lemmas \ref{weight1} and \ref{weight3}, we deduce that $\omega\in
A_{i(\Phi)-\varepsilon_0}$ for some $\varepsilon_0>0$. Because $i(\Phi)>1$,
we assume that $i(\Phi)-\varepsilon_0>1$ and
\begin{equation}\label{js}
	\left\{
	\begin{aligned}
		\int_{\Omega}
		(\Phi\circ G)(|Dh|)\,\omega(x)\,dx
		&\leq \delta^{I(\Phi)-\varepsilon_0},
		\\[0.15cm]
		\int_{\Omega}
		(\Phi\circ G)(|F|)\,\omega(x)\,dx
		&\leq \delta^{I(\Phi)-\varepsilon_0},
		\\[0.15cm]
		\int_{\Omega}
		(\Phi\circ G)(|D\psi_i|)\,\omega(x)\,dx
		&\leq \delta^{I(\Phi)-\varepsilon_0},
		\qquad i=1,2,
	\end{aligned}
	\right.
\end{equation}
where $\delta \in(0,\frac{1}{8})$ will be selected later. Using $\Phi \in \Delta_2$, we have
$$\Phi(Kt)\rs c_{\varepsilon_0} \min\left\lbrace K^{i(\Phi)-\varepsilon_0}, K^{I(\Phi)-\varepsilon_0}\right\rbrace\Phi(t), \quad for\ all \  K,t\geqslant0. $$
By $\omega \in A_{i(\Phi)-\varepsilon_0} \subset A_{I(\Phi)-\varepsilon_0}$ and   \holder's    inequality, we calculate that
\begin{eqnarray*}
&&\int_{\Omega \cap \left\lbrace G(|Dh|)\geqslant1\right\rbrace }G(|Dh|)dx \\
&\ls& \left( \int_{\Omega \cap \left\lbrace G(|Dh|)\geqslant1\right\rbrace }G(|Dh|)^{i(\Phi)-\varepsilon_0}\omega(x)  dx \right)^{\frac{1}{i(\Phi)-\varepsilon_0}} \left( \omega^{\frac{-1}{i(\Phi)-\varepsilon_0-1}}(\Omega) \right) ^{\frac{i(\Phi)-\varepsilon_0-1}{i(\Phi)-\varepsilon_0}} \\
&\ls& c \left(\int_{\Omega} (\Phi\circ G)(|Dh|)\omega(x)dx\right) ^{\frac{1}{i(\Phi)-\varepsilon_0}}  \ls c\delta^{\frac{I(\Phi)-\varepsilon_0}{i(\Phi)-\varepsilon_0}} \ls  c\delta.
\end{eqnarray*}
Similarly,
\begin{eqnarray*}
\int_{\Omega \cap \left\lbrace G(|Dh|)<1\right\rbrace }G(|Dh|)dx \ls c \left(\int_{\Omega} (\Phi\circ G)(|Dh|)\omega(x)dx\right) ^{\frac{1}{I(\Phi)-\varepsilon_0}} \ls c\delta.
\end{eqnarray*}
Combining  the above two inequalities, we obtain
\begin{equation*}
\int_{\Omega }G(|Dh|)dx \ls c \delta.
\end{equation*}
Similarly,
\begin{equation*}
\int_{\Omega }G(|F|)dx \ls c \delta \ \ \ \ \  \& \ \ \ \ \  \int_{\Omega }G(|D\psi_i|)dx \ls c \delta, \quad i=1,2.
\end{equation*}
We take comparison function $v=h$ in the inequality \eqref{u} and make use of  \eqref{u1}, Lemma \ref{gyoung}, \eqref{a(x)4} and Lemma \ref{vg} to find
\begin{align*}
	\int_{\Omega} G(|Du|)\,\mathrm{d}x
	&\leq c\int_{\Omega}a(x,Du)\cdot Du\,\mathrm{d}x
	\nonumber\\
	&\leq c\int_{\Omega}a(x,Du)\cdot Dh\,\mathrm{d}x
	+c\int_{\Omega}b(x,F)\cdot(Du-Dh)\,\mathrm{d}x
	\nonumber\\
	&\leq c\varepsilon_1\int_{\Omega}G(|Du|)\,\mathrm{d}x
	+c(\varepsilon_1)\int_{\Omega}G(|Dh|)\,\mathrm{d}x
	\nonumber\\
	&\quad
	+c\varepsilon_2\int_{\Omega}G(|Du|)\,\mathrm{d}x
	+c(\varepsilon_2)\int_{\Omega}G(|F|)\,\mathrm{d}x.
	\label{energy-estimate}
\end{align*}
Now we choose $\varepsilon_1, \varepsilon_2$ small enough to get
\begin{equation}\label{dudel}
\int_{\Omega} G(|Du|)dx \ls c\int_{\Omega} G(|F|)+G(|Dh|)dx \ls c \delta.
\end{equation}
{
Next, for a fixed \( \varepsilon > 0 \), we choose \( C \) and \( \delta \) as in Lemma \ref{deome}. Using the boundedness of \( \Omega \), we can find a finite set of points \(\{z_j\}_{j=1}^l \subset \bar{\Omega}\) and a ball \( B \) satisfying
\[
\bar{\Omega} \subset \bigcup_{j=1}^l B_r(z_j) \subset B, \quad \text{with } r = \frac{R}{2000}.
\]
Then we apply the weak $(1,1)$ type estimate
\[
|\{x \in \mathbb{R}^n : \mathcal{M}(f)(x) > k\}| \leq \frac{c}{k} \int_{\mathbb{R}^n} |f(x)| dx,
\quad\forall\, k > 0
\]
and use \eqref{dudel} to obtain 
\begin{eqnarray*}
|\{x \in \Omega : \mathcal{M}(G(|Du|))(x) > C\}| &\leq& \frac{c}{C} \int_{\Omega} G(|Du|) dx< c\delta \\
&<& \left\{ \frac{\varepsilon}{d^2} \left( \frac{|B_r(z_j)|}{|B|} \right)^{i(\Phi)} \right\}^{\frac{1}{\nu}} |B|,
\end{eqnarray*}
where \( d \) and \( \nu \) are as in Lemma \ref{weight2}, and \( \delta > 0 \) is additionally chosen small enough to ensure the last inequality.  Now, from Lemma \ref{weight2},
it follows that
\begin{equation*}
\omega(\{x \in \Omega : \mathcal{M}(G(|Du|))(x) > C\}) < \frac{\varepsilon}{d} \left( \frac{|B_r(z_j)|}{|B|} \right)^{i(\Phi)} \omega(B)
\end{equation*}
and
\[
\omega(B) \leq d \left( \frac{|B|}{|B_r(z_j)|} \right)^{i(\Phi)} \omega(B_r(z_j)), \quad\forall\,
j\in\{1, \ldots, l\}.
\]
Combining these estimates yields
\[
\omega(\{x \in \Omega : \mathcal{M}(G(|Du|))(x) > C\}) < \varepsilon \omega(B_r(z_j)), \quad\forall\,
j\in\{1, \ldots, l\},
\]
which, together with Lemma \ref{deome}, further implies that, for
\(
\varepsilon_1 = 20^{ni(\Phi)}[\omega]_{i(\Phi)}^{2}\varepsilon,
\)
\begin{eqnarray*}
 \omega(\{x \in \Omega : \mathcal{M}(G(|Du|))(x) > C^k\}) &\leq& \varepsilon_1^k \omega(\{x \in \Omega : \mathcal{M}(G(|Du|))(x) > 1\})  \\
&&\hspace{0.1cm}+ \sum_{i=1}^k \varepsilon_1^i \omega(\{x \in \Omega : \mathcal{M}(G(|F|))(x) > \frac{C^{k-i} \delta}{4}\}) \\
&&\hspace{0.1cm}+ \sum_{i=1}^k \varepsilon_1^i \omega(\{x \in \Omega : \mathcal{M}(G(|D\psi_1|))(x) > \frac{C^{k-i} \delta}{4}\})
\\
&&\hspace{0.1cm}+ \sum_{i=1}^k \varepsilon_1^i \omega(\{x \in \Omega : \mathcal{M}(G(|D\psi_2|))(x) > \frac{C^{k-i} \delta}{4}\})
\\
&&\hspace{0.1cm}+ \sum_{i=1}^k \varepsilon_1^i \omega(\{x \in \Omega : \mathcal{M}(G(|Dh|))(x) > \frac{C^{k-i} \delta}{4}\})
\end{eqnarray*}
From $\Phi \in \Delta_2$, we deduce that
$$\Phi(Kt)\ls c_{\varepsilon_0} \max\left\lbrace K^{i(\Phi)-\varepsilon_0}, K^{I(\Phi)-\varepsilon_0}\right\rbrace\Phi(t), \quad for \ all \ K,t\geqslant0. $$
We then use these estimates to compute as follows:
\begin{eqnarray*}
&&\sum_{k=1}^\infty \Phi(C^k) \omega(\{x \in \Omega : \mathcal{M}(G(|Du|))(x) > C^k\}) \\
& \leq& \sum_{k=1}^\infty \varepsilon_1^k \Phi(C^k) \omega(\{x \in \Omega : \mathcal{M}(G(|Du|))(x) > 1\}) \\
&&\hspace{0.1cm}+ c \sum_{i=1}^{\infty} (\varepsilon_1 C^{I(\Phi) - \varepsilon_0})^i \sum_{k=i}^{\infty} \Phi(C^{k-i}) \omega(\{x \in \Omega : \mathcal{M}(G(|F|))(x) > \frac{C^{k-i} \delta}{4}\}) \\
&&\hspace{0.1cm}+ c \sum_{i=1}^{\infty} (\varepsilon_1 C^{I(\Phi) - \varepsilon_0})^i \sum_{k=i}^{\infty} \Phi(C^{k-i}) \omega(\{x \in \Omega : \mathcal{M}(G(|D\psi_1|))(x) > \frac{C^{k-i} \delta}{4}\}) \\
&&\hspace{0.1cm}+ c \sum_{i=1}^{\infty} (\varepsilon_1 C^{I(\Phi) - \varepsilon_0})^i \sum_{k=i}^{\infty} \Phi(C^{k-i}) \omega(\{x \in \Omega : \mathcal{M}(G(|D\psi_2|))(x) > \frac{C^{k-i} \delta}{4}\}) \\
&&\hspace{0.1cm}+ c \sum_{i=1}^{\infty} (\varepsilon_1 C^{I(\Phi) - \varepsilon_0})^i \sum_{k=i}^{\infty} \Phi(C^{k-i}) \omega(\{x \in \Omega : \mathcal{M}(G(|Dh|))(x) > \frac{C^{k-i} \delta}{4}\}) \\
&\leq& c\bigl(3\Phi(1)\omega(\Omega) + S_1 + S_2+ S_3+ S_4 \bigr) \sum_{k=1}^{\infty} (\varepsilon_1 C^{I(\Phi) - \varepsilon_0})^k,
\end{eqnarray*}
where
\begin{equation*}
	\left\{
	\begin{aligned}
		S_1
		&:= \sum_{j=1}^{\infty}\Phi(C^j)\,
		\omega\left(
		\left\{
		x\in\Omega:
		\mathcal{M}\bigl(G(|F|)\bigr)(x)
		> \frac{C^j\delta}{4}
		\right\}
		\right),
		\\[0.15cm]
		S_2
		&:= \sum_{j=1}^{\infty}\Phi(C^j)\,
		\omega\left(
		\left\{
		x\in\Omega:
		\mathcal{M}\bigl(G(|D\psi_1|)\bigr)(x)
		> \frac{C^j\delta}{4}
		\right\}
		\right),
		\\[0.15cm]
		S_3
		&:= \sum_{j=1}^{\infty}\Phi(C^j)\,
		\omega\left(
		\left\{
		x\in\Omega:
		\mathcal{M}\bigl(G(|D\psi_2|)\bigr)(x)
		> \frac{C^j\delta}{4}
		\right\}
		\right),
		\\[0.15cm]
		S_4
		&:= \sum_{j=1}^{\infty}\Phi(C^j)\,
		\omega\left(
		\left\{
		x\in\Omega:
		\mathcal{M}\bigl(G(|Dh|)\bigr)(x)
		> \frac{C^j\delta}{4}
		\right\}
		\right).
	\end{aligned}
	\right.
\end{equation*}
It follows from Lemmas \ref{weight3}  and \ref{weight4}, \eqref{js} that
\begin{align*}
S_1 &\leq c \int_{\Omega} \Phi\left(\frac{G(|F|)}{\delta}\right) \omega(x)\,dx
 \leq c \int_{\Omega} \Phi\left(\frac{G(|F|)}{\delta}\right) \omega(x)\,dx \\
    &\leq \frac{c}{\delta^{I(\Phi) - \varepsilon_0}} \int_{\Omega} (\Phi \circ G)(|F|) \omega(x)\,dx
\leq c,
\end{align*}
and, similarly,
\begin{equation*}
	\left\{
	\begin{aligned}
		S_2
		&\leq
		\frac{c}{\delta^{I(\Phi)-\varepsilon_0}}
		\int_{\Omega}
		(\Phi\circ G)(|D\psi_1|)\,\omega(x)\,dx
		\leq c,
		\\[0.2cm]
		S_3
		&\leq
		\frac{c}{\delta^{I(\Phi)-\varepsilon_0}}
		\int_{\Omega}
		(\Phi\circ G)(|D\psi_2|)\,\omega(x)\,dx
		\leq c,
		\\[0.2cm]
		S_4
		&\leq
		\frac{c}{\delta^{I(\Phi)-\varepsilon_0}}
		\int_{\Omega}
		(\Phi\circ G)(|Dh|)\,\omega(x)\,dx
		\leq c.
	\end{aligned}
	\right.
\end{equation*}
Thus,
\[
\sum_{k=1}^{\infty} \Phi(C^k) \omega(\{x \in \Omega : \mathcal{M}(G(|Du|))(x) > C^k\}) \leq c \sum_{k=1}^{\infty} (\varepsilon_1 C^{I(\Phi) - \varepsilon_0})^k.
\]
Choosing \(\varepsilon > 0\) sufficiently small such that \(\varepsilon_1 C^{I(\Phi) - \varepsilon_0} < 1\), Lemma \ref{weight4} yields
\[
\int_{\Omega} \Phi(\mathcal{M}(G(|Du|))) \omega(x)\,dx \leq c.
\]
We finally conclude that
\begin{equation*}
\int_{\Omega}(\Phi\circ G)(|Du|)\omega(x)dx \ls c,
\end{equation*}
under the hypothesis \eqref{js}.
}

Now consider
\begin{equation*}
	\left\{
	\begin{aligned}
		\overline{u}
		&= \frac{u}{M},\qquad
		\overline{F}
		= \frac{F}{M},\qquad
		\overline{h}
		= \frac{h}{M},\qquad
		\overline{\psi_i}
		= \frac{\psi_i}{M},
		\quad i=1,2,
		\\[0.2cm]
		\overline{g}(t)
		&= g(Mt),\qquad
		\overline{G}(t)
		= \frac{G(Mt)}{M},\qquad
		\overline{\Phi}(t)
		= \frac{\Phi(Mt)}{M},
		\\[0.2cm]
		\overline{a}(x,\eta)
		&= a(x,M\eta),\qquad
		\overline{b}(x,\eta)
		= b(x,M\eta),
	\end{aligned}
	\right.
\end{equation*}
where
\begin{equation*}
	\begin{aligned}
		M:=\delta^{\varepsilon_0-I(\Phi)}
		\Bigg[
		&\int_{\Omega}(\Phi\circ G)(|F|)\,\omega(x)\,\mathrm{d}x
		+\int_{\Omega}(\Phi\circ G)(|D\psi_1|)\,\omega(x)\,\mathrm{d}x
		\\
		&+\int_{\Omega}(\Phi\circ G)(|D\psi_2|)\,\omega(x)\,\mathrm{d}x
		+\int_{\Omega}(\Phi\circ G)(|Dh|)\,\omega(x)\,\mathrm{d}x
		\Bigg].
	\end{aligned}
\end{equation*}
By \cite[Lemma 3.5]{cho}, we have $i(\overline{\Phi})=i(\Phi)$ and $I(\overline{\Phi})=I(\Phi)$,
which shows that
\begin{equation*}
	\left\{
	\begin{aligned}
		\int_{\Omega}
		(\overline{\Phi}\circ\overline{G})(|\overline{F}|)
		\,\omega(x)\,dx
		&=
		\frac{1}{M}
		\int_{\Omega}
		(\Phi\circ G)(|F|)
		\,\omega(x)\,dx
		\leq
		\delta^{I(\overline{\Phi})-\varepsilon_0},
		\\[0.2cm]
		\int_{\Omega}
		(\overline{\Phi}\circ\overline{G})(|D\overline{\psi_i}|)
		\,\omega(x)\,dx
		&=
		\frac{1}{M}
		\int_{\Omega}
		(\Phi\circ G)(|D\psi_i|)
		\,\omega(x)\,dx
		\leq
		\delta^{I(\overline{\Phi})-\varepsilon_0},
		\quad i=1,2,
		\\[0.2cm]
		\int_{\Omega}
		(\overline{\Phi}\circ\overline{G})(|D\overline{h}|)
		\,\omega(x)\,dx
		&=
		\frac{1}{M}
		\int_{\Omega}
		(\Phi\circ G)(|Dh|)
		\,\omega(x)\,dx
		\leq
		\delta^{I(\overline{\Phi})-\varepsilon_0}.
	\end{aligned}
	\right.
\end{equation*}
This further implies that
\begin{equation*}
\int_{\Omega}(\overline{\Phi}\circ \overline{G})(|D\overline{u}|)\omega(x)dx \ls c,
\end{equation*}
which completes the proof of \eqref{uff}.
\end{proof}

\hskip\parindent   
 
\section*{Acknowledgments}
The authors are supported by
  Sichuan Natural Science Foundation Youth Fund Project (Grant No.2025ZNSFSC0799 ), Natural Science Foundation of Hubei Province (JCZRYB202401287) and the  National Natural Science Foundation of China
  (Grant No.~12401122,12571226).

\textbf{Data Availability Statement} This manuscript has no associated data.

\textbf{ Conflict Of Interest Statement} We confirm that we do not have any conflict of interest.

\end{document}